\documentclass[12pt]{article}

\usepackage{mathrsfs}
\usepackage{amssymb, amsthm, amsfonts, amsxtra, amsmath}
\usepackage{latexsym}
\usepackage[all]{xy}
\usepackage{graphics}
\usepackage{makeidx}
\usepackage{rotating}
\usepackage{syntonly}
\usepackage{nomencl}
\usepackage{pdfpages}

\usepackage{mathrsfs}
\usepackage{amssymb, amsthm, amsfonts, amsxtra, amsmath}
\usepackage{latexsym}

\usepackage{bbm}

\usepackage{parskip}
\makeatletter 
\def\thm@space@setup{%
 \thm@preskip=\parskip \thm@postskip=0pt
}
\def\th@remark{%
  \thm@headfont{\itshape}%
  \normalfont 
  \thm@preskip\parskip \thm@postskip=0pt
}
\makeatother

\theoremstyle{change}

\newtheorem{Theorem}{Theorem}[section]
\newtheorem{Def}[Theorem]{Definition}
\newtheorem{Lem}[Theorem]{Lemma}
\newtheorem{Prop}[Theorem]{Proposition}
\newtheorem{Cor}[Theorem]{Corollary}

\newcommand{\id}{\mathrm{id}}

\newcommand{\Spec}{\mathrm{Spec}}

\newcommand{\N}{\mathbb{N}}
\newcommand{\Z}{\mathbb{Z}}
\newcommand{\R}{\mathbb{R}}
\newcommand{\C}{\mathbb{C}}

\newcommand{\Hsp}{\mathcal{H}}

\newcommand{\rd}{\mathrm{d}}

\newcommand{\varphiH}{\psi}

\newcommand{\alg}{\mathrm{alg}}

\DeclareMathOperator{\Mor}{\mathrm{Mor}}

\begin{document}

\title{A field of quantum upper triangular matrices}
\author{K. De Commer and M. Flor\'{e}}

\date{}
\author{Kenny De Commer\thanks{Department of Mathematics, Vrije Universiteit Brussel, VUB, B-1050 Brussels, Belgium, email: {\tt kenny.de.commer@vub.ac.be}. This work was supported by FWO grant G.0251.15N. This work is also part of the project supported by the NCN-grant 2012/06/M/ST1/00169.}\;\, and 
Matthias Flor\'{e}\thanks{Department of Mathematics, Vrije Universiteit Brussel, VUB, B-1050 Brussels, Belgium, email: {\tt matthias.flore@vub.ac.be}.} }
\maketitle

\begin{abstract} \noindent We show that the duals of Woronowicz's quantum $SU(2)$-groups converge, within the operator algebraic setting, to the group of special upper triangular 2-by-2 matrices with positive diagonal.
\end{abstract}

\section*{Introduction}

\emph{Quantum groups}, and in particular quantized enveloping algebras and their dual quantum function algebras, have manifestations in different contexts, for example purely algebraic, at either a formal or a scalar parameter, or operator algebraic, at either the C$^*$-algebraic or von Neumann algebraic level. The correspondence between the scalar, algebraic approach and the operator algebraic approach is at this moment well understood in the compact semisimple case, see e.g.~ the basic reference works \cite{ChP94, KS97, Maj95}. Also the correspondence between the algebraic formal approach and the operator algebraic approach works nicely in the compact semisimple setting, using the formalism of continuous fields of C$^*$-algebras \cite{NT11}. 

However, it is known since Drinfel'd's fundamental work  \cite{Dri88} that formal quantized enveloping algebras can also be seen as formal quantized function algebras, reflecting the Poisson duality which is present in the quasi-classical limit. The operator algebraic version of this `dual field' seems not to have received much attention yet in the compact semisimple case, although it offers an interesting case of a family of discrete `quantum' structures converging to a non-discrete, continuous classical structure. 

In this article, we want to start amending this situation. Our main aim will be to show how the family of duals of the quantum groups $SU_q(2)$ \cite{Wor87a} can be embedded into a continuous field of quantum groups over the interval $(0,1\rbrack$ such that, for $q=1$, the fiber is isomorphic to the group of special upper triangular 2-by-2 matrices with positive diagonal. 

Our point of departure will be the work of Blanchard \cite{Bla96}, where a detailed study of fields of quantum groups and fields of multiplicative unitaries was made. Combined with the work of Woronowicz \cite{Wor95} and results from \cite{KL11}, this will allow us to obtain our main theorems in quite a computation-free way. We note that in \cite[Section 7.2]{Bla96}, it was shown how the $C(SU_q(2))$ can be made into a field of C$^*$-algebras which (as quantum groups) converge to the group C$^*$-algebra of the $az+b$-group for $q\rightarrow 1$. The abstract duality theory of \cite{Bla96} can then also be used to yield a field of `quantum upper triangular matrices', but we prefer to present a more direct and explicit approach, which also seems more amenable to  generalizations to higher rank situations. We refrain from making a direct comparison with the field obtained in \cite[Section 7.2]{Bla96}. 

The precise content of our article is as follows. 

In the \emph{first section}, we construct the underlying field of C$^*$-algebras $C_0(G)$ over $(0,1\rbrack$, making use of the notion of \emph{crossed product with a partial automorphism} \cite{Exe94}. Continuity of our field will follow from a general result concerning continuous fields of such crossed products. We also give a concrete field of representations for our field of C$^*$-algebras. 

In the \emph{second section}, we use the Baaj-Woronowicz theory of affiliated ope-rators to show that $C_0(G)$ possesses a coassociative comultiplication $\Delta$ from $C_0(G)$ into $C_0(G)\underset{C_0((0,1\rbrack)}{\otimes} C_0(G)$, which is moreover bisimplifiable in the sense of \cite{Bla96}. This easily lets us conclude that $(C_0(G),\Delta)$ has an associated conti-nuous field of multiplicative unitaries. 
 
We assume in this article that all Hilbert spaces are separable. We also assume that all C$^*$-algebras (denoted by the letters $A,B,C,\ldots$) are separable, except for those which obviously can't be, such as multiplier C$^*$-algebras of non-unital C$^*$-algebras. In particular, all locally compact Hausdorff spaces are assumed second countable. 

We will call \emph{representation} of a C$^*$-algebra $A$ a non-degenerate $^*$-representation of $A$ on a Hilbert space. For $A,B$ two C$^*$-algebras, we write $\Mor(A,B)$ for the set of all non-degenerate $^*$-homomorphisms $A\rightarrow M(B)$. We call \emph{embedding} any injective map in $\Mor(A,B)$. By subalgebra we mean any non-degenerate inclusion $A\subseteq B$ of C$^*$-algebras. 

We write $\otimes$ for the tensor product between Hilbert spaces or for the minimal tensor product between C$^*$-algebras.

\section{A field of C$^*$-algebras}

\subsection{$C_0(Y)$-algebras}

Let $Y$ be a locally compact Hausdorff space. Recall that a \emph{$C_0(Y)$-algebra} is a C$^*$-algebra $A$ together with an embedding of $C_0(Y)$ into the center of $M(A)$. 

\begin{Def} The \emph{localisation} $A_y$ at $y\in Y$ is the quotient of $A$ by the closed ideal $I_y$ consisting of all $fa$ with $f\in C_0(Y\setminus \{y\})$ and $a\in A$.

\end{Def}

For $a\in A$, the image of $a$ in $A_y$ will be written $a_y$.

\begin{Def} A $C_0(Y)$-algebra $A$ is called a \emph{continuous field of C$^*$-algebras over $Y$} if the map \[y\mapsto \|a_y\|\] is continuous on $Y$ for all $a\in A$.
\end{Def} 

Let $\mathscr{F}$ be a Hilbert $C_0(Y)$-module. For $y\in Y$, we write $\mathscr{F}_y$ for the Hilbert space obtained as the separation-completion of $\mathscr{F}$ with respect to the seminorm $\|\xi\|_y = \sqrt{\langle \xi,\xi\rangle(y)}$. Write $\xi_y$ for the image of $\xi\in \mathscr{F}$ in $\mathscr{F}_y$.

\begin{Def}(\cite[D\'{e}finition 2.11]{Bla96})  A \emph{$C_0(Y)$-representation} on $\mathscr{F}$ of a $C_0(Y)$-algebra $A$  consists of a $C_0(Y)$-linear element  $\pi\in \Mor(A,\mathcal{K}(\mathscr{F}))$, where $\mathcal{K}(\mathscr{F})$ denotes the C$^*$-algebra of compact operators on $\mathscr{F}$.
\end{Def} 

In this situation, $\pi$ factorizes into representations $\pi_y$ of $A_y$ on the $\mathscr{F}_y$. 

\begin{Def}(\cite[D\'{e}finition 2.11]{Bla96})\label{DefFaithField}  A $C_0(Y)$-representation $\pi$ of a $C_0(Y)$-algebra $A$ is called a \emph{field of faithful representations} if the $\pi_y$ are faithful representations for all $y\in Y$.
\end{Def} 

This implies that $\pi$ itself is faithful since by \cite[Proposition 2.8]{Bla96} \begin{equation}\label{EqNormSup}\|a\|=\sup_{y\in Y} \|a_y\|,\qquad \forall a\in A.\end{equation}  

\begin{Theorem}\label{TheoBlanFF} A $C_0(Y)$-algebra $A$ is a continuous field of C$^*$-algebras if and only if it admits a field of faithful representations. 
\end{Theorem}
\begin{proof} This is \cite[Th\'{e}or\`{e}me 3.3.(1)$\Leftrightarrow$(4)]{Bla96}.
\end{proof}

\begin{Lem}\label{LemField1} Let $H,Y$ be locally compact Hausdorff spaces. Let $\pi: H \twoheadrightarrow Y$ be continuous and open, and write $H_y = \pi^{-1}(y)$ for $y\in Y$.

Then $C_0(Y)\subseteq M(C_0(H))$ makes $C_0(H)$ into a continuous field of C$^*$-algebras over $Y$, and the natural maps $C_0(H)_y \rightarrow C_0(H_y)$ are isomorphisms.
\end{Lem} 
\begin{proof} See \cite[Proposition 3.14]{Bla96} and the discussion preceding it. 
\end{proof}

If $A$ is a $C_0(Y)$-algebra, we say that an automorphism $\alpha$ of $A$ is a \emph{$C_0(Y)$-automorphism} if $\alpha$ is $C_0(Y)$-linear. For $B\subseteq A$ an inclusion of C$^*$-algebras, we call \emph{conditional expectation of $A$ onto $B$} a $B$-bilinear completely positive map $F: A \rightarrow B$ of norm $1$.

\begin{Lem}\label{LemField3}Let $A$ be a $C_0(Y)$-algebra. Let $G$ be a compact Hausdorff group acting continuously on $A$ by $C_0(Y)$-automorphisms. Let $A^G$ be the fixed point subalgebra of $A$ with respect to $\alpha$. 

Then $A^G$ is a $C_0(Y)$-algebra, the action of $G$ descends to each fiber $A_y$, and \[(A_y)^G \cong (A^G)_y.\] Moreover, $A$ is a continuous field of C$^*$-algebras if and only if $A^G$ is a continuous field of C$^*$-algebras.
\end{Lem}

\begin{proof} Denote the action of $G$ on $A$ by $\alpha$. It is immediate that $A^G$ is a $C_0(Y)$-subalgebra of $A$, and that $\alpha$ descends to each fiber $A_y$. 

For each $y\in Y$, we obtain a $^*$-homomorphism $\rho: (A^G)_y \rightarrow (A_y)^G$. Let \[F: A \rightarrow A^G,\quad a \mapsto \int_G \alpha_g(a)\rd g,\] so that $F$ is a faithful conditional expectation of $A$ onto $A^G$. For $b_y \in (A_y)^G$, we have $F(b)\in A^G$ with \[\rho(F(b)_y) =  \left(\int_G\alpha_g(b)\rd g\right)_y = \int_G\alpha_g(b_y)\rd g=   b_y,\] hence $\rho$ is surjective. On the other hand, if $a\in A^G$ and $a_y=0$ as an element of $A_y$, then $a = a'f$ with $a'\in A$ and $f\in C_0(Y\setminus\{y\})$. But then $a = F(a')f$, hence $a_y=0$ as an element of $(A^G)_y$. This shows that $\rho$ is injective, hence a $^*$-isomorphism. 

From the above, it follows that for each $y$ we have a faithful conditional expectation \[F_y: A_y \rightarrow A_y^G,\quad a_y\mapsto \int_G \alpha_g(a_y)\rd g,\] leading to a commutative diagram \[\xymatrix{ A \ar[r]\ar[d]_{F} & A_y\ar[d]^{F_y} \\ A^G\ar[r] & A_y^G.}\] We can then conclude the last statement of the Lemma from the equivalence \cite[Th\'{e}or\`{e}me 3.3.(1) $\Leftrightarrow$ (3)]{Bla96}.
\end{proof}

\subsection{Crossed products with partial automorphisms}\label{SubSecPartAut}

We recall some of the results of \cite{Exe94} in the case of commutative C$^*$-algebras. 

Let $H$ be a locally compact space with two open subsets $H^{(1)},H^{(-1)}\subseteq H$ together with a homeomorphism \[\theta: H^{(1)}\rightarrow H^{(-1)}.\] We call $\Theta = (\theta,H^{(1)},H^{(-1)})$ a \emph{partial automorphism} of $H$. 

Since $H^{(1)}$ and $H^{(-1)}$ are open sets in $H$, we can consider $C_0(H^{(1)})\subseteq C_0(H)$ and $C_0(H^{(-1)})\subseteq C_0(H)$ as closed  ideals. Then $\theta$ gives rise to a $^*$-isomorphism \[\theta: C_0(H^{(-1)})\rightarrow C_0(H^{(1)}),\quad f\mapsto f\circ \theta,\] for which we use the same notation $\theta$. The triple $\Theta = (\theta,C_0(H^{(-1)}),C_0(H^{(1)}))$ is a \emph{partial automorphism} of $C_0(H)$ in the sense of \cite{Exe94}. This leads to the crossed product $C_0(H)\rtimes_{\Theta}\Z$ in the sense of \cite[Definition 3.7]{Exe94}.

To get a more concrete description of $C_0(H)\rtimes_{\Theta}\Z$, consider more generally for $k\in \Z$ the open set $H^{(k)}\subseteq H$ which is the domain of $\theta^k$, and consider $\theta^k$ as a homeomorphism \[\theta^k: H^{(k)}\rightarrow H^{(-k)}.\] Consider for $k\in \Z$ the sets \[\mathscr{N}_k = \{\textrm{formal symbols }s^kf\mid f\in C_0(H^{(k)})\}.\] Then $\oplus_{k\in \Z} \mathscr{N}_k$ can be turned into a $^*$-algebra by the formulas \[(s^kf)(s^lg) = s^{k+l}\theta^l(f\theta^{-l}(g)),\quad (s^kf)^* = s^{-k} \theta^{-k}(f^*).\] The C$^*$-algebra $C_0(H)\rtimes_{\Theta}\Z$ is by definition the universal C$^*$-algebraic envelope of $\oplus_{k\in \Z} \mathscr{N}_k$. Moreover, by \cite[Proposition 3.11]{Exe94} the map \[\oplus_{k\in \Z} \mathscr{N}_k \rightarrow C_0(H)\rtimes_{\Theta}\Z\] is injective. In particular, we have an inclusion $C_0(H) \rightarrow C_0(H)\rtimes_{\Theta} \Z$. In the following, we write $s^0f = f$ for $f \in C_0(H)$. 

Note that $\mathscr{N} = \mathscr{N}_1$ is naturally a Hilbert bimodule over $C_0(H)$ by \[\langle x,y\rangle_r = x^*y,\quad \langle x,y\rangle_l = xy^*.\] This Hilbert bimodule completely determines $C_0(H)\rtimes_{\Theta}\Z$.

\begin{Prop}\label{PropDefRel} The C$^*$-algebra $C_0(H)\rtimes_{\Theta}\Z$ is the universal C$^*$-algebra generated by a copy of the C$^*$-algebra $C_0(H)$ and the vector space $\mathscr{N}$ with defining relations, for $f\in C_0(H)$ and $x,y\in \mathscr{N}$, \[x\cdot f = xf,\quad f\cdot x = fx,\quad  x^*\cdot y = \langle x,y\rangle_r,\quad x\cdot y^* = \langle x,y\rangle_l.\]
\end{Prop} 
\begin{proof} This follows from \cite[Example 3.2]{AEE98}.
\end{proof} 

Denote by $S^1 =\{z\in \C\mid |z|=1\} \subseteq \C$ the circle group. 

\begin{Def}(\cite[Section 3]{AEE98}) The \emph{dual action} $\beta$ of $S^1$ on $C_0(H)\rtimes_{\Theta}\Z$ is defined by \[\beta_z(x) = z^k x,\quad x\in \mathscr{N}_k,z\in S^1.\] 
\end{Def} 

\begin{Prop} The $k$-th spectral subspace with respect to $\beta$ is $\mathscr{N}_k$. Moreover, for $k> 0$, \[\mathscr{N}_k = \mathscr{N}_1^k,\quad \mathscr{N}_{-k} = \mathscr{N}_{-1}^{k}.\]
\end{Prop}  
The powers designate the closure of the linear span of $k$-fold products. 
\begin{proof} See \cite[Proposition 3.11, Proposition 4.7 and Proposition 4.8]{Exe94}. 
\end{proof} 

\begin{Def} We define $F$ to be the faithful conditional expectation \[F:C_0(H)\rtimes_{\Theta} \Z \rightarrow C_0(H),\quad x\mapsto \int_{S^1} \beta_{z}(x)\rd z.\] 

\end{Def} 
\subsection{Fields of crossed products}\label{SecFieldCross}

Assume that $Y,H$ are locally compact spaces with $\pi: H \twoheadrightarrow Y$ continuous and open. Assume further that $\Theta = (\theta,H^{(1)},H^{(-1)})$ is a partial automorphism of $H$, and assume that $\pi \circ \theta = \pi_{\mid H^{(1)}}$. Write \[H_y^{(k)} = H^{(k)}\cap \pi^{-1}(y).\] By our assumptions, $\theta$ restricts to a homeomorphism $\theta_y: H_y^{(1)}\rightarrow H_y^{(-1)}$, and we obtain a partial homeomorphism $\Theta_y = (\theta_y,H^{(1)}_y,H^{(-1)}_y)$ of $H_y$. 

\begin{Prop}\label{PropFieldCross} The C$^*$-algebra $C_0(H) \rtimes_{\Theta} \Z$ is a continuous field of C$^*$-algebras over $Y$ with fibers $C_0(H_y)\rtimes_{\Theta_y} \Z$.
\end{Prop} 
\begin{proof} By our assumptions, the natural embedding $C_0(Y)\rightarrow M(C_0(H) \rtimes_{\Theta} \Z)$ takes values in the center, so that $C_0(H)\rtimes_{\Theta} \Z$ is a $C_0(Y)$-algebra. Moreover, the dual action is clearly an action by $C_0(Y)$-automorphisms. Then $(C_0(H)\rtimes_{\Theta}\Z)^{S^1} = C_0(H)$ is a continuous field of C$^*$-algebras by Lemma \ref{LemField1}, and $C_0(H)\rtimes_{\Theta}\Z$ is a continuous field of C$^*$-algebras  by Lemma \ref{LemField3}. 

From the defining relations and Tietze's extension theorem, it is clear that we have a surjective $^*$-homomorphism \[(C_0(H)\rtimes_{\Theta} \Z)_y \rightarrow C_0(H_y)\rtimes_{\Theta_y}\Z.\] On the other hand, we also have, by Lemma \ref{LemField3}, a commuting square \[\xymatrix{ (C_0(H)\rtimes_{\Theta} \Z)_y \ar[r] \ar[d]^{\textrm{c.e.}}& C_0(H_y)\rtimes_{\Theta_y}\Z \ar[d]^{\textrm{c.e.}} \\ C_0(H)_y \ar[r]^{\cong}& C_0(H_y).}\]
From the faithfulness of the conditional expectations, it follows that the top map is also isometric,  hence a $^*$-isomorphism. 
\end{proof} 

\subsection{A field of quantum triangular matrices} \label{SecFieldQTM}

We will make use of the following notation: for $0<q<1$ and $b>0$, we define \[(0,b\rbrack_q = \{q^jb \mid j\in \N\}.\] For $0<a<b$ with $a \in q^{\N}b$, we write \[\lbrack a,b\rbrack_q = (0,b\rbrack_q\setminus (0,qa\rbrack_q = \{b,qb,q^2b,\ldots,q^{-1}a,a\}.\] We also write, for $0<a<b$, \[(0,b\rbrack_1 = (0,b\rbrack,\qquad \lbrack a,b\rbrack_1 = \lbrack a,b\rbrack.\]

\begin{Def}\label{NotSetX} We define \[H =  \left\{(q,t,a)\mid 0 < q\leq 1, t\in (0,q\rbrack_q, a\in \left\lbrack \sqrt{\frac{t}{q}},\sqrt{\frac{q}{t}}\;\right\rbrack_q\right\}.\] 
\end{Def} 

\begin{Lem} The space $H$ is locally compact (with the trace topology).
\end{Lem} 
\begin{proof} This follows since $H$ is closed in $(\R^+_0)^3$. 
\end{proof}

We write \[H^{(1)} = \{(q,t,a)\in H\mid t+t^{-1} -qa^2-q^{-1}a^{-2} \neq 0\},\]\[H^{(-1)} = \{(q,t,a)\in H\mid t+t^{-1}-q^{-1}a^2-qa^{-2}\neq 0\}.\] We then have  a homeomorphism \[\theta: H^{(1)} \rightarrow H^{(-1)},\quad (q,t,a)\mapsto (q,t,qa),\] which moreover commutes with the open, surjective projection \[\pi: H \rightarrow (0,1\rbrack,\quad (q,t,a)\mapsto q.\]  We are hence in the situation of Section \ref{SecFieldCross}.

\begin{Def}  We define \[C_0(G)= C_0(H)\rtimes_{\Theta}\Z, \quad C_0(G_q) = C_0(H_q)\rtimes_{\Theta_q} \Z.\] 
\end{Def}

From Proposition \ref{PropFieldCross}, we obtain the following corollary.

\begin{Cor} The $C_0((0,1\rbrack)$-algebra $C_0(G)$ is a continuous field of C$^*$-algebras over $(0,1\rbrack$ with fibers $C_0(G_q)$. \end{Cor}

In the following, we give a more concrete description of these fibers. 

\begin{Lem}\label{LemIrRepq} For $0<q<1$, we have $C_0(G_q) \cong c_0$-$\oplus_{n=1}^{\infty} B(\C^n)$. 
\end{Lem} 

\begin{proof}Consider the Hilbert space $\Hsp = \oplus_{n=1}^{\infty} \C^n$, whose basis elements we denote by $e_{k}^{(n)}$ with $0\leq k < n$. Write $e_{k,l}^{(n)}$ for the standard matrix units of $c_0$-$\oplus_{n=1}^{\infty} B(\C^n) \subseteq B(\Hsp)$. Write $D \subseteq c_0$-$\oplus_{n=1}^{\infty} B(\C^n )$ for the subalgebra of diagonal matrices. Then we have an isomorphism \[\rho: C_0(H_q) \cong D,\quad f\mapsto \rho(f)= \sum_{n,k} f(q^n,q^kq^{\frac{1-n}{2}})e_{k,k}^{(n)}.\] Let \[S: \Hsp \rightarrow \Hsp,\quad e_{k}^{(n)} \mapsto e_{k+1}^{(n)},\]  where the latter vector is interpreted as zero when ill-defined. Then clearly $S^k  x \in c_0$-$\oplus_{n=1}^{\infty} B(\C^n)$ for $x\in D$, and it is easily checked by the defining relations of $C_0(G_q)$ that the map \[s^k f \mapsto S^k \rho(f)\] extends to a $^*$-homomorphism \[\rho: C_0(G_q) \rightarrow  c_0\textrm{-}\!\oplus_{n=1}^{\infty} B(\C^n).\] As $H_q$ is a discrete set, one sees by using Dirac functions that $\rho$ is a surjective map. On the other hand, $\rho$ intertwines the faithful conditional expectation $F: C_0(G_q) \rightarrow C_0(H_q)$ with the natural conditional expectation \[c_0\textrm{-}\!\oplus_{n=1}^{\infty} B(\C^n) \rightarrow D,\quad e_{k,l}^{(n)}\mapsto \delta_{k,l}e_{k,k}^{(n)}.\] It follows that $\rho$ is a $^*$-isomorphism. 
\end{proof} 

\begin{Lem}\label{LemIrRep1} We have $C_0(G_1) \cong C_0(\C\times \R_0^+)$. 
\end{Lem} 
\begin{proof} Consider on $\C\times \R_0^+$ the circle action \[\alpha_z((n,a)) = (zn,a).\] Then the map \[\rho: H_1 \rightarrow (\C\times \R_0^+)/S^1,\quad (t,a)\mapsto ((t+t^{-1}-a^2-a^{-2})^{1/2},a)S^1\] is a homeomorphism. Note further that in this case, $H_1^{(k)}$ is the same set $H_1'$ for all $k\neq 0$, and $\theta$ the identity map on $H_1'$. The map $\rho$ restricts to a homeomorphism \[H_1' \rightarrow (\C_0\times \R_0^+)/S^1 \cong \R_0^+\times \R_0^+.\] We can then consider $C_0(H_1) \subseteq C_0(\C\times \R_0^+)$ as $S^1$-invariant functions, and similarly $C_0(H_1') \subseteq C_0(\C_0\times \R_0^+) \subseteq C_0(\C\times \R_0^+)$. 

Consider on $\C\times \R_0^+$ the measurable functions \[r^{\pm}:\left\{\begin{array}{lllll} (n,a)&\mapsto& (n/|n|)^{\pm 1} &\textrm{for}&n\neq 0,\\ (n,a)&\mapsto& 0& \textrm{for}& n=0.\end{array}\right.\] Write $r^{0}$ for the identity function on $\C\times \R_0^+$. Then $r^nf \in C_0(\C\times \R_0^+)$ for $f\in C_0(H_1')$, and we hence obtain a $^*$-homomorphism \[\rho: C_0(G_1) \rightarrow C_0(\C\times \R_0^+),\quad s^kf \mapsto r^k\rho(f).\] By Stone-Weierstrass, this map is surjective. As this map intertwines the $S^1$-actions and is isometric on the fixed point algebra, it is also injective, hence an isomorphism. 
\end{proof}

\subsection{A field of faithful representations}

We keep the notation from the previous section. According to Theorem \ref{TheoBlanFF}, $C_0(G)$ must admit a field of faithful representations. In this section, we describe a concrete instance of such a field. 

\begin{Def} We define $I = H \times S^1$, and $I_q = H_q\times S^1$.
\end{Def} 

Consider on $I_1$ the measure\[\mu_1=\frac{1}{2\pi}(t^{-2}-1)a\rd t\rd a\rd \theta\] and associated functional \[\varphiH_1(f) = \frac{1}{2\pi}\int_0^{2\pi} \int_0^{1}\int_{\sqrt{t}}^{\frac{1}{\sqrt{t}}} f(t,a,e^{i\theta})(t^{-2}-1)a\rd a\rd t\rd \theta,\qquad f\in C_c(I_1),\] where $C_c(I_1)$ denotes the space of continuous functions on $I_1$ with compact support.

Let us further endow $C_c(I_q)$ with the positive functional \[\varphiH_q(f) = \frac{ (1-q)^2}{2\pi}\sum_{(t,a)\in H_q} (t^{-1}-t)a^{2}\int_{0}^{2\pi} f(t,a,e^{i\theta}) \rd \theta.\]

\begin{Lem}\label{LemContFieldHilb} For $f\in C_c(I)$ and $0< q\leq 1$, write $f_q \in C_c(I_q)$ for \[f_q(t,a,z) = f(q,t,a,z).\]Then the map \[\varphiH: C_c(I) \rightarrow C_c((0,1\rbrack),\quad f\mapsto \left(q\mapsto \varphiH_q(f_q)\right)\] is well defined, faithful and positive. 
\end{Lem}
\begin{proof} 
Let $f\in C_c(I)$. Define \[E: C_c(I) \rightarrow C_c(H), \quad E(f)(q,t,a) = \frac{1}{2\pi} \int_0^{2\pi} f(q,t,a,e^{i\theta})\rd\theta.\] Then $\varphiH$ satisfies \[\varphiH = \varphiH\circ E.\] It is hence sufficient to prove that the restriction of $\varphiH$ to $C_c(H)$ has range in $C_c((0,1\rbrack)$. 

But for $f\in C_c(H)$ and $0< q<1$,

\[
 \varphiH(f)(q) = \varphiH_q(f_q)= (1-q)^2\sum_{(t,a)\in H_q} f(q,t,a) (t^{-1}-t)a^{2}.\]
 
 Define $\int_0^1 g(x)\rd_qx= (1-q)\sum_{n=0}^{\infty} g(q^n)q^n$. Then we can write 
\begin{align*} \varphiH(f)(q) &= (1-q)^2
\sum_{n=0}^{\infty}\sum_{k=-\frac{n}{2}}^{\frac{n}{2}} f(q,q^{n+1},q^k) (q^{-n-1}-q^{n+1})q^{2k} \\ 
&= (1-q)^2\sum_{n=0}^{\infty}\sum_{k=0}^{n} f(q,q^{n+1},q^{k-\frac{n}{2}}) (q^{-n-1}-q^{n+1})q^{2k-n}\\ 
&= (1-q)^2\sum_{k=0}^{\infty} \sum_{n=k}^{\infty}   f(q,q^{n+1},q^{k-\frac{n}{2}}) (q^{-n-1}-q^{n+1})q^{2k-n} \\ 
&= (1-q)^2\sum_{k=0}^{\infty} \sum_{n=0}^{\infty}   f(q,q^{n+k+1},q^{\frac{k-n}{2}}) (q^{-n-k-1}-q^{n+k+1})q^{k-n} \\ 
&= (1-q)^2\sum_{k=0}^{\infty} \sum_{n=0}^{\infty}   f(q,q^{n+k+1},q^{\frac{k-n}{2}}) (q^{-n-k-1}-q^{n+k+1})q^{-2n}q^kq^n \\ 
&= \int_0^1\int_0^1 f(q,qxy,\sqrt{y/x}) (q^{-1}x^{-1}y^{-1}-qxy)x^{-2} \rd_qx\rd_qy.\end{align*} 

We easily see that $\psi(f)$ is continuous on $(0,1)$, with moreover \[\varphiH(f)(q) \underset{{q\rightarrow 1}}{\rightarrow} \varphiH'_1(f_1)\] where \[\varphiH'_1(g) = \int_0^1\int_0^1 g(xy,\sqrt{y/x}) (x^{-1}y^{-1}-xy)x^{-2} \rd x\rd y,\quad g\in C_c(H_1).\] 
However, with $H_1'=\{(t,a)\in H_1\mid \sqrt{t}<a<\frac{1}{\sqrt{t}}\}$, we have a diffeomorphism \[\Phi: (0,1)^2\rightarrow H_1',\quad (x,y)\mapsto (xy,\sqrt{y/x})\] with $|J_{\Phi}(t,a)| = a$. Hence \[\varphiH'_1(g) =  \int_0^1 \int_{\sqrt{t}}^{\frac{1}{\sqrt{t}}} g(t,a)(t^{-2}-1)a\rd a \rd t = \varphiH_1(g).\]This proves that $\varphiH$ has range in $C_c((0,1\rbrack)$. 

Of course, the positivity of $\varphiH$ is immediate, as is the faithfulness, since each $\varphiH_q$ is faithful. \end{proof}

\begin{Def} We define $\mathscr{I}$ to be the Hilbert $C_0((0,1\rbrack)$-module obtained by completing $C_c(I)$ with respect to the $C_0((0,1\rbrack)$-valued inner product \[\langle f,g\rangle = \varphiH(f^*g).\] 
\end{Def}

\begin{Lem}\label{LemFibHilb} For all $0<q\leq 1$, we have a natural identification \[\mathscr{I}_q = L^2(I_q,\varphiH_q).\] Moreover, under this identification, \[C_c(I)_q = C_c(I_q).\]
\end{Lem} 
\begin{proof} By definition, we obtain an isometric map $C_c(I)_q \rightarrow C_c(I_q)$ sending $f_q$ to the restriction of $f$ to $I_q$. By Tietze's extension theorem, it is surjective. The lemma follows. 
\end{proof} 

The natural projection $I \rightarrow H$ leads to the natural embedding \[C_0(H)\subseteq C_0(I)\] as $S^1$-independent functions. We further denote \[I^{(k)} = H^{(k)}\times S^1 \subseteq I.\] Then $\Theta$ extends to a partial homeomorphism of $I$ by \[\theta: I^{(1)}\rightarrow I^{(-1)}, \quad (q,t,a,z) \mapsto (q,t,qa,z).\]Denote by $Z$ the function \[Z: I \rightarrow \C,\quad (q,t,a,z)\mapsto z.\]

\begin{Lem}\label{LemFaithG} There is a field of faithful representations  \[\pi\in \Mor(C_0(G), \mathcal{K}(\mathscr{I}))\] such that \[\pi(f)h= fh,\qquad \pi(sg)h =Z \theta^{-1}(gh)\] for $f\in C_0(H), g\in C_0(H^{(1)})$ and $h\in C_c(I)$. 

\end{Lem} 

\begin{proof} It is clear that $\pi$ is a well-defined representation of $C_0(H)$. Furthermore, one easily checks that \[\psi_q(\theta^{-1}(f)) = q^2\psi_q(f),\quad f\in C_c(I)\cap C_0(I^{(1)}),\] so that $\pi(sg)$ extends to a bounded operator on $\mathscr{I}$. An easy computation shows that the defining relations in Proposition \ref{PropDefRel} are satisfied for the $\pi(f)$ and $\pi(sg)$, so that there exists a representation $\pi\in \Mor( C_0(G),\mathcal{K}(\mathscr{I}))$ as above. 

To see that it is a field of faithful representations, note that the localisations $\pi_q$ are faithful on the $C_0(H_q)$. But each $L^2(I_q)$ also carries a continuous representation of $S^1$ by \[\alpha_z(f)(t,a,w) =f(t,a,zw),\]  and
\[\pi_q(\beta_z(a_q)) = U_z\pi_q(a_q)U_z^*,\quad z\in S^1,a_q\in C_0(G_q).\] It follows that each $\pi_q$ is faithful. 
\end{proof}

\section{A field of locally compact quantum groups}

\subsection{Affiliated operators}

Let $A$ be a C$^*$-algebra, and $\mathscr{F}$ a right Hilbert $A$-module. Recall that an \emph{unbounded operator} on $\mathscr{F}$ is an $A$-linear operator \[T:\mathscr{D}(T) \subseteq \mathscr{F}\rightarrow \mathscr{F}\] with $A$-invariant dense domain $\mathscr{D}(T)$. One calls an operator $T$ \emph{semiregular}\footnote{One sometimes requires also that $T$ is closed, but it will be more convenient for us not to require this from the outset.}  if it has a densily defined adjoint operator $T^*$. In this case $T$ is automatically closable, and $T^*$ is a closed semiregular operator on $\mathscr{F}$. Also the closure of $T$ is again semiregular. One calls $T$ \emph{regular} if $T$ is closed and semiregular and $1+T^*T$ is invertible.  

\begin{Def}(\cite{BaJ83,NaW92}) Let $A$ be a C$^*$-algebra. We call the set $A^{\eta}$ of all regular operators on $A$, considered as a right Hilbert $A$-module over itself, the set of elements \emph{affiliated} with $A$. When $T\in A^{\eta}$, we write $T\eta A$. 
\end{Def}

For $T\in A^{\eta}$, the element \[z_T = T(1+T^*T)^{-1/2} \in M(A)\] is called the \emph{$z$-transform} of $T$. If $\pi\in \Mor(A,B)$, there exists a unique element $\pi(T)\in B^{\eta}$ such that $\pi(z_T) = z_{\pi(T)}$. 

In general, the affiliation relation is not easy to check. In \cite{KL11}, it was shown how the affiliation relation can be checked \emph{locally}. The following definition makes sense by \cite[Section 2.4]{KL11}.
\begin{Def} Let $A$ be a C$^*$-algebra, $T$ a semiregular operator on $A$, and $\pi$ a representation of $A$ on a Hilbert space $\Hsp$. Then there exists a unique closable, densily defined operator $T_{\pi}$ on $\Hsp$ with domain $\pi(\mathscr{D}(T))\Hsp$ such that \[T_{\pi}\pi(a)\xi = \pi(Ta)\xi,\quad a\in \mathscr{D}(T),\xi\in \Hsp.\] 
\end{Def}

It is easy to see that one then has $(T^*)_{\pi} \subseteq (T_{\pi})^*$. The following theorem is proven by combining \cite[Theorem 5.10]{KL11} with \cite[Theorem 4.2.1.(1)$\Leftrightarrow$(3) and Theorem 3.3.(1)$\Leftrightarrow$(2)]{KL11}.

\begin{Theorem}[\cite{KL11}]\label{TheoKLOr} Let $A$ be a C$^*$-algebra. Then a closed semiregular operator $T$ on $A$ is regular (and hence affiliated to $A$) if and only if $(T_{\pi})^*$ is the closure of $(T^*)_{\pi}$ for all irreducible representations $\pi$ of $A$. Moreover, a right $A$-submodule $\mathscr{D}\subseteq \mathscr{D}(T)$ is a core for $T$ if and only if $\mathscr{D}_{\pi}=\pi(\mathscr{D})\Hsp$ is a core for $T_{\pi}$ for each irreducible representation $\pi$ of $A$. 
\end{Theorem}

We will need the following particular case. It provides us with a class of C$^*$-algebras for which semiregularity already implies regularity.

\begin{Theorem}\label{TheoKL}  Let $A$ be a C$^*$-algebra such that $\pi(A)$ is unital for each irreducible representation $\pi$ of $A$. Then any closed, semiregular operator $T$ on $A$ is regular, and any right $A$-invariant dense domain of $T$ is a core for $T$. 
\end{Theorem} 

\begin{proof} Let $\pi$ be an irreducible representation of $A$. Then left multiplication by $\pi(T)$ gives a well-defined, semiregular operator on the unital C$^*$-algebra $\pi(A)$ with domain $\pi(\mathscr{D}(T))$. However, as $\pi(\mathscr{D}(T))$ is a dense right ideal, it equals $\pi(A)$. Hence $\pi(T)$ is a bounded operator. The theorem hence follows from Theorem \ref{TheoKLOr}.
\end{proof}

\subsection{Quantum generating family for $C_0(G)$}

In this whole section, we keep the notation from Section \ref{SecFieldQTM}.
Consider the following coordinate functions on $H$, \[Q(q,t,a) = q,\quad \Omega(q,t,a) = t+t^{-1}, \quad A(q,t,a) = a\] and \[|N|^2 = \Omega - QA^2-Q^{-1}A^{-2},\quad |N^*|^2 =  \Omega - Q^{-1}A^2-QA^{-2}.\] Then $Q$ is strictly positive and bounded, with bounded inverse, $A$ is positive, unbounded and invertible, and $\Omega$ is unbounded and self-adjoint. It is also easily checked that $|N|^2$ and $|N^*|^2$ are positive, with $H^{(1)}$ (resp.~ $H^{(-1)}$) the zero set of $|N|^2$ (resp.~ $|N^*|^2$). We denote by $|N|$ (resp.~ $|N^*|$) the unique positive root of $|N|^2$ (resp.~ $|N^*|^2$).

More generally, for $k\geq 1$ we define $|N^k|,|(N^*)^k|\in C(H)$ as the functions \[|N^k| =  \left(\prod_{l=0}^{k-1} (\Omega - Q^{2l+1}A^2-Q^{-2l-1}A^{-2})\right)^{1/2},\]\[|(N^*)^k|= \left(\prod_{l=0}^{k-1} (\Omega - Q^{-2l-1}A^2-Q^{2l+1}A^{-2})\right)^{1/2}.\]

\begin{Def} For $f\in C_c(H)$ and $k\geq 1$, we define \[N^kf = s^k(|N^k|f),\qquad (N^*)^kf = s^{-k}(|(N^*)^k|f)\] as elements in $C_0(G)$. 
\end{Def}

Note that this is meaningful since  $|N^k|f \in C_0(H^{(k)})$ (reading $N^{-1}= N^*$). 

\begin{Def} We define $C_c(G)$ to be the two-sided $^*$-ideal of $C_0(G)$ ge-nerated by $C_c(H)$. 
\end{Def} 

\begin{Lem} The $^*$-algebra $C_c(G)$ is dense in $C_0(G)$, and \begin{equation}\label{EqComp} C_c(G) = C_c(H)C_0(G) = C_0(G)C_c(H).\end{equation} 
Moreover, $\{N^kf \mid f\in C_c(H)\}$ is dense in $\mathscr{N}_k$.  
\end{Lem} 

In \eqref{EqComp}, the right hand sides consist of linear combinations of products of elements in the corresponding sets.

\begin{proof} Density of $C_c(G)$ in $C_0(G)$ follows immediately since $C_c(H)$ is dense in $C_0(H)$ and the inclusion $C_0(H) \subseteq C_0(G)$ is non-degenerate.

For $f\in C_c(H)$ and $g\in C_0(H^{(k)})$, we have that \[\|N^k f-s^k g\| = \||N^k|f-g\|.\] It follows that $\{N^kf \mid f\in C_c(H)\}$ is dense in $\mathscr{N}_k$. 

To prove \eqref{EqComp}, pick $f\in C_c(H)$ and choose $\varepsilon>0$ such that $f(q,t,a)=0$ for all $(q,t,a)$ with $\mathrm{min}\{q,t\}<\varepsilon$. Choose $h\in C_c(H)$ such that $h(q,t,a) = 1$ for all $(q,t,a)$ with $\mathrm{min}\{q,t\}\geq \varepsilon$.  Then for $g\in C_0(H^{(k)})$, it follows that \[h(s^kg)f = s^k(\theta^k(h\theta^{-k}(gf))) = (s^k g)f.\] Hence $hC_0(G)f = C_0(G)f$ and so \[C_c(G) = C_0(G)C_c(H)C_0(G) \subseteq C_c(H)C_0(G)C_c(H)C_0(G) \subseteq C_c(H)C_0(G).\] We obtain $C_c(G) = C_c(H)C_0(G)$. As $C_c(G)$ is $^*$-invariant, the other identity in \eqref{EqComp} follows.
\end{proof} 

With a small modification, the proof shows in fact that $C_c(G)$ has local units in $C_c(H)$, that is, for each finite collection $x_1,\ldots, x_n\in C_c(G)$ there exists $e\in C_c(H)$ with $x_i = ex_i = x_ie$ for all $i$. In particular, we can write any element in $C_c(G)$ as $fx$ with $f\in C_c(H)$ and $x\in C_0(G)$.  

We want to interpret the symbol $N$ as an operator affiliated with $C_0(G)$. 

\begin{Prop}\label{PropExN} There exists an operator $N$ affiliated with $C_0(G)$ and with (invariant) core $C_c(G)$ such that \begin{equation}\label{EqDefN} N(fx) = (Nf)x,\quad f\in C_c(H),\quad x\in C_0(G).\end{equation} Moreover, $C_c(G)$ is also a core for $N^*$ and \begin{equation}\label{EqDefNstar} N^*(fx) = (N^*f)x,\quad f\in C_c(H),\quad x\in C_0(G).\end{equation}
\end{Prop}

\begin{proof} For $f,g \in C_c(H)$ and $x,y\in C_0(G)$, an easy computation shows that \[ (gy)^* (Nf)x = ((N^*g)y)^* fx.\] Hence there exists a well-defined semiregular operator $N$ with core $C_c(G)$ satisfying \eqref{EqDefN} and with $C_c(G) \subseteq \mathscr{D}(N^*)$ satisfying \eqref{EqDefNstar}. 

Now since any irreducible representation of $C_0(G)$ must factor over some $C_0(G_q)$, it follows from Lemma \ref{LemIrRepq} and Lemma \ref{LemIrRep1} that all irreducible representations of $C_0(G)$ are finite dimensional. By Theorem \ref{TheoKL} we conclude that $N$ is affiliated to $C_0(G)$ and $C_c(G)$ a core for $N^*$. 
\end{proof} 

As the coordinate functions $Q^{\pm 1}, A^{\pm 1} \in C(H)$ are affiliated with $C_0(H)$, we can also interpret them as elements affiliated with $C_0(G)$. It is in fact clear that $C_c(G)$ is also an invariant core for each of these operators. Note that on this common core, these operators then satisfy the relations \[AN = QNA,\quad  \lbrack N,N^*\rbrack = (Q-Q^{-1})(A^2-A^{-2}),\] which, for $Q$ considered a fixed positive real number strictly smaller than $1$, are precisely the relations for the quantum group $U_q(\mathfrak{su}(2))$ (up to rescaling). 

We will need a small extension of the above result. 

\begin{Prop}\label{PropExT} There exists a unique regular operator $T \eta M_2(C_0(G))$ such that $M_2(C_c(G))$ is a core for $T$ and $T = \begin{pmatrix} A & N \\ 0 & A^{-1}\end{pmatrix}$ on $M_2(C_c(G))$. Moreover, $M_2(C_c(G))$ is also a core for $T^*$, with $T^* = \begin{pmatrix} A & 0 \\ N^* & A^{-1}\end{pmatrix}$ on $M_2(C_c(G))$.  
\end{Prop} 

\begin{proof} The proof is identical to that of Proposition \ref{PropExN}.
\end{proof} 

The above constructions can also be performed on the localisations $C_0(G_q)$, leading to the $^*$-algebra $C_c(G_q)$ and the affiliated operators $N_q,A_q\eta C_0(G_q)$ and $T_q\eta M_2(C_0(G_q))$. Since $C_c(H)_q = C_c(H_q)$, we also have that $C_c(G_q)$ is the localisation of $C_c(G)$ at $q$, and $N_q,A_q,T_q$ are the localisations of $N,A,T$.

\subsection{A quantum generating family}

We aim to show that the coordinate function $Q$, together with the operator $T\eta M_2(C_0(G))$ defined in Proposition \ref{PropExT}, is a quantum generating family for $C_0(G)$ in the sense of \cite{Wor95}. We will need some preliminaries. 

\begin{Def}(\cite[Definition 4.1]{Wor95}) Let $A,C$ be C$^*$-algebras. We say $A$ is \emph{generated} by (the quantum family of unbounded operators) $T\in (C\otimes A)^{\eta}$ if the following holds: for any Hilbert space $\Hsp$, any representation $\pi$ of $A$ on $\Hsp$ and any C$^*$-subalgebra $B$ of $B(\Hsp)$, the affiliation $(\id\otimes \pi)(T)\in(C\otimes B)^{\eta}$ implies that $\pi\in \Mor(A,B)$.
\end{Def} 

Here the separability of $A$ and $C$ is crucial to ensure that this definition satisfies the following, to be expected property.

\begin{Lem}\label{LemSepRep} Let $A,C$ be C$^*$-algebras. Let $T\in (C\otimes A)^{\eta}$ generate $A$. Then $T$ separates representations of $A$: if $\pi,\rho$ are representations of $A$ on a Hilbert space $\Hsp$ with $(\id\otimes \pi)T = (\id\otimes \rho)T$, then $\pi=\rho$.
\end{Lem}
\begin{proof} This follows from \cite[Theorem 6.2.(I)$\Rightarrow$(II)]{Wor95}.
\end{proof}

We will need the following criterion to know whether a quantum family of operators generates a C$^*$-algebra, see \cite[Theorem 4.2 and Remark 4.4]{Wor95}.

\begin{Lem}\label{LemCritWor} Assume that $A,C$ are C$^*$-algebras, and $T\in (C\otimes A)^{\eta}$. Assume that the following two conditions are satisified:
\begin{enumerate}
\item The operator $T$ separates representations of $A$.
\item There exists an element $r\in A$ such that for any representation $\pi$ of $A$ on a Hilbert space $\Hsp$ and any C$^*$-subalgebra $B$ of $B(\Hsp)$, the affiliation $(\id\otimes \pi)T \in (C\otimes B)^{\eta}$ implies $\pi(r)B$ and $B\pi(r)$ contained and dense in $B$ (and, in particular, $\pi(r)\in M(B)$).  
\end{enumerate}
Then $A$ is generated by $T$. 
\end{Lem} 

Assume now that $Y$ is a locally compact Hausdorff space. Note that if $A$ is a $C_0(Y)$-algebra and $T\eta A$, then we can make sense of $T_y \eta A_y$ since the localisation map is an element of $\Mor(A,A_y)$. Extensions of morphisms defined on generators can then be created from local information as follows. 

\begin{Prop}\label{PropGen} Let $A,B$ be $C_0(Y)$-algebras. Let $C$ be a finite-dimensional C$^*$-algebra, and let $T\in (C\otimes A)^{\eta}$ and $S \in (C\otimes B)^{\eta}$. Suppose $A$ is generated by $T$, and suppose that for all $y\in Y$ there exists $\varphi_{y}\in \Mor(A_{y},B_{y})$ such that $(\id\otimes \varphi_y)T_y=S_{y}$. Then there exists a unique $\varphi\in \Mor(A,B)$ such that $(\id\otimes \varphi)(T)=S$.
\end{Prop}

Here we use the obvious identification $(C\otimes A)_y \cong C\otimes A_y$, viewing $C\otimes A$ as a $C_0(Y)$-algebra in the natural way.

\begin{proof} Since $A,B$ are separable, we can find, by the identity \eqref{EqNormSup} under Definition \ref{DefFaithField}, an at most countable subset $Y_0\subseteq Y$ such that the natural non-degenerate maps \[A \overset{i_A}{\hookrightarrow} M(\oplus_{y\in Y_0}{A_{y}}),\quad B \overset{i_B}{\hookrightarrow} M(\oplus_{y\in Y_0}{B_y})\] are faithful. We then want to show there exists a unique arrow $\varphi$ making the following diagram commute,

\begin{displaymath}
\xymatrix{A\ar@{^{(}->}[r]^-{i_{A}} \ar@{-->}[d]^{\varphi} & M(\oplus_{y\in Y_0}{A_{y}})\ar@{}[r]|{\supseteq}\ar[d]^{\phi=\oplus_{y\in Y_0}{\varphi_{y}}} & \prod_{y\in Y_0}^b{M(A_{y})}\\M(B)\ar@{^{(}->}[r]^-{i_{B}} & M(\oplus_{y\in Y_0}{B_{y}}) \ar@{}[r]|{\supseteq}& \prod_{y\in Y_0}^b{M(B_{y})}, }
\end{displaymath}
where $\prod^b$ denotes the bounded direct product of C$^*$-algebras. Also, note that the map $\phi: \oplus_{y\in Y_0}A_y \rightarrow M(\oplus_{y\in Y_0} B_y)$ is non-degenerate, hence extends uniquely to $M(\oplus_{y\in Y_0}A_y)$. 

By direct computation, and using that $C$ passes through direct products by finite-dimensionality, we have \[(\id\otimes \phi\circ i_{A})z_{T}=(\id\otimes \phi)\prod_{y\in Y_0}{z_{T_{y}}}=\prod_{y\in Y_0}{z_{S_{y}}}=(\id\otimes i_{B})z_{S},\] so that $(\id\otimes \phi\circ i_A)T = (\id\otimes i_B)S$. Hence we can apply \cite[Proposition 4.5]{Wor95} to conclude that there exists $\varphi$ as in the statement of the proposition. Its uniqueness follows from Lemma \ref{LemSepRep}.
\end{proof}

In the next lemma, we will need the disintegration of a representation of a $C_0(Y)$-algebra along $Y$, see \cite[Chapter 8]{Dix77}. 
	
\begin{Lem}\label{LemSpecLoc} Assume $S\in C_0(Y)$ separates points of $Y$, so $Y \cong \Spec(S)$. Let $A$ be a $C_0(Y)$-algebra. Let $C$ be a finite-dimensional C$^*$-algebra, and $T\in (C\otimes A)^{\eta}$. Then $(S,T) \eta (\C\oplus C)\otimes A$ separates representations of $A$ if and only if $T_y\in (C\otimes A_y)^{\eta}$ separates representations of $A_y$ for each $y$. 
\end{Lem} 
\begin{proof} Assume first that $T_y$ separates representations of $A_y$ for each $y$. Let $\pi,\rho$ be two representations of $A$ on a Hilbert space $\Hsp$, and assume that $\pi(S) = \rho(S)$ and $(\id\otimes \pi)T = (\id\otimes \rho)T$. Then $\pi(f) = \rho(f)$ for all $f\in C_0(Y)$. Hence we have disintegrations of the form \[\Hsp = \int_{Y}^{\oplus} \Hsp_y \rd \mu(y), \quad \pi =  \int_{Y}^{\oplus} \pi_y \rd \mu(y),\quad \rho = \int_{Y}^{\oplus} \rho_y \rd \mu(y)\] for some Borel measure $\mu$ on $Y$, where $\pi_y$ and $\rho_y$ are representations of $A$ which factor over $A_y$.

As $(\id\otimes \pi)T = (\id\otimes\rho)T$, it follows that $(\id\otimes \pi_y)T_{y} = (\id\otimes \rho_y)T_{y}$ for almost all $y$. As $T_{y}$ separates the representations of $A_y$, this implies, by Lemma \ref{LemSepRep}, $\pi_y = \rho_y$ for almost all $y$, and hence $\pi = \rho$.

Conversely, if $(S,T)$ separates representations of $A$ and $y\in Y$, then any two representations $\pi,\rho$ of $A_y$ with $(\id\otimes \pi)T_y = (\id\otimes \rho)T_y$ lift to representations of $A$ which are equal on $(S,T)$, and hence $\pi = \rho$.  
\end{proof} 

Let us now return to $C_0(G)$ and the operator $T$ from Proposition \ref{PropExT}. 

\begin{Theorem}\label{TheoExT} The couple \[(Q,T)\eta (\C\oplus M_2(\C))\otimes C_0(G)\] is a quantum generating family for $C_0(G)$. 
\end{Theorem}
 \begin{proof}
It is immediate that the localisation of $T$ coincides with $T_q$ as defined beneath Proposition \ref{PropExT}. As the $T_q$ separate representations of $C_0(G_q)$ for each $q$ by \cite[Section 4, Example 6]{Wor95}, it follows by Lemma \ref{LemSpecLoc} that $(Q,T)$ separates representations of $C_0(G)$. 

Exactly the same argument as in \cite[Section 4, Example 6]{Wor95} shows that there exists an element $r\in A$ satisfying Condition 2 in Lemma \ref{LemCritWor}. It follows that $(Q,T)$ generates $C_0(G)$. 
\end{proof} 

\subsection{A field of comultiplications}\label{SecFieldCom}

Again we begin this section with some general preliminaries.

\begin{Def}(\cite[D\'{e}finition 3.19]{Bla96}) Let $A$ and $B$ be $C_0(Y)$-algebras over some locally compact space $Y$, and assume they both admit faithful $C_0(Y)$-representations on respective Hilbert $C_0(Y)$-modules $\mathscr{F}$ and $\mathscr{E}$. We define the \emph{$C_0(Y)$-tensor product} $A\underset{C_0(Y)}{\otimes} B$ to be the image of the natural representation of $A\underset{\mathrm{max}}{\otimes} B$ on the interior tensor product  $\mathscr{F}\underset{C_0(Y)}{\otimes} \mathscr{E}$. 
\end{Def} 

One can show \cite[Proposition 3.20]{Bla96} that the $C_0(Y)$-tensor product is independent (up to canonical isomorphism) of the choice of $\mathscr{F}$ and $\mathscr{E}$. It is also immediate from the definition that $A\underset{C_0(Y)}{\otimes} B$ is then again a $C_0(Y)$-algebra admitting a faithful $C_0(Y)$-representation, and that $\underset{C_0(Y)}{\otimes}$ satisfies associati-vity in a natural way. 

When one of the $C_0(Y)$-algebras is nuclear (as a C$^*$-algebra), one can say a bit more. Note that by \cite[Proposition 3.23]{Bla96} a $C_0(Y)$-algebra $A$ is nuclear if and only if all fibers $A_y$ are nuclear. 

\begin{Prop}(\cite[Corollaire 3.17, Proposition 3.25 and Corollaire 3.26]{Bla96})\label{PropTensCont} Assume $A,B$ are continuous fields of C$^*$-algebras over $Y$, and assume $A$ is nuclear. Then $A\underset{C_0(Y)}{\otimes} B$ is the universal C$^*$-envelope of the algebraic $C_0(Y)$-balanced tensor product $^*$-algebra $A\overset{\alg}{\underset{C_0(Y)}{\otimes}} B$, and $A\underset{C_0(Y)}{\otimes} B$ is a continuous field of C$^*$-algebras over $Y$ with \[(A\underset{C_0(Y)}{\otimes} B)_y = A_y\otimes B_y,\qquad \forall y\in Y.\]
\end{Prop}

Let us now return to $C_0(G)$. Note that $C_0(G)$, having a faithful conditional expectation onto $C_0(H)$, is nuclear. Hence $C_0(G) \underset{C_0((0,1\rbrack)}{\otimes} C_0(G)$ is a continuous field of C$^*$-algebras with fibers $C_0(G_q)\otimes C_0(G_q)$. 

\begin{Def} We define \[C_0(G\underset{(0,1\rbrack}{\times} G) = C_0(G) \underset{C_0((0,1\rbrack)}{\otimes} C_0(G).\] 
\end{Def} 

Write $H\underset{(0,1\rbrack}\times H$ for the closed subset \[H\underset{(0,1\rbrack}{\times}  H = \{(q,t,a,q,t',a')\mid (q,t,a),(q,t',a') \in H_q\}\subseteq H\times H.\] Then one has an identification \[C_0(H\underset{(0,1\rbrack}\times H) \cong C_0(H)\underset{C_0((0,1\rbrack)}{\otimes} C_0(H),\] leading to an embedding \[C_0(H\underset{(0,1\rbrack}\times H) \hookrightarrow C_0(G\underset{(0,1\rbrack}{\times} G).\] 

In fact, $C_0(G\underset{(0,1\rbrack}{\times}  G)$ has an action of $S^1\times S^1$ by \[\beta_{z,w}(x\otimes y) = \beta_z(x)\otimes \beta_w(y),\quad x,y\in C_0(G),\] with $C_0(H\underset{(0,1\rbrack}{\times} H)$ as its fixed point algebra. One may also see $C_0(G\underset{(0,1\rbrack}{\times} G)$ as a crossed product by a partial action of the group $\Z^2$ in the sense of \cite{Mcc95}. The $(m,n)$th-spectral subspace can be identified with \[C_0(G\underset{(0,1\rbrack}{\times}G)_{(m,n)} = (s^m\otimes s^n)C_0(H^{(m)}\underset{(0,1\rbrack}{\times}H^{(n)}),\] where elements in the right hand side are easily defined by approximation with elementary tensors.

\begin{Def} We define $C_c(G\underset{(0,1\rbrack}{\times} G)$ as the ideal in  $C_0(G\underset{(0,1\rbrack}{\times} G)$ generated by  $C_c(H\underset{(0,1\rbrack}{\times} H)$.
\end{Def} 
 
The same argument which was used to prove \eqref{EqComp} shows that \[C_c(G\underset{(0,1\rbrack}{\times} G) = C_0(G\underset{(0,1\rbrack}{\times} G)C_c(H\underset{(0,1\rbrack}{\times} H) = C_c(H\underset{(0,1\rbrack}{\times} H)C_0(G\underset{(0,1\rbrack}{\times} G).\]

If now $f\in C_c(H\underset{(0,1\rbrack}{\times}H)$, we have for example \[(A\otimes |N|)f \in C_0(H\underset{(0,1\rbrack}{\times} H^{(1)}),\quad (|N|\otimes A^{-1})f \in C_0(H^{(1)}\underset{(0,1\rbrack}{\times} H).\]

This allows us to make sense of \[(A\otimes N)f = (1\otimes s)((A\otimes |N|)f),\quad (N\otimes A^{-1})f = (s\otimes 1)((|N|\otimes A^{-1})f)\] inside $C_0(G\underset{(0,1\rbrack}{\times}G)$.  Similarly, operators $A\otimes N^*$ and $N^*\otimes A^{-1}$ can be defined.

\begin{Prop}\label{PropExNDelta} There exists an operator $\Delta(N)$ affiliated with $C_0(G\underset{(0,1\rbrack}{\times} G)$ and with (invariant) core $C_c(G\underset{(0,1\rbrack}{\times} G)$ such that \begin{equation}\label{EqDefDelN} \Delta(N)(fx) = ((A\otimes N+ N\otimes A^{-1})f)x.\end{equation} for all $f\in C_c(H\underset{(0,1\rbrack}{\times} H)$ and $x\in C_0(G\underset{(0,1\rbrack}{\times} G)$. Moreover, $C_c(G\underset{(0,1\rbrack}{\times} G)$ is also a core for $\Delta(N)^*$, and
  \[\Delta(N)^*(fx) = ((A\otimes N^*+N^*\otimes A^{-1})f)x.\]  
\end{Prop}

\begin{proof} Since any irreducible representation of $C_0(G\underset{(0,1\rbrack}{\times}G)$ splits over some $C_0(G_y)\otimes C_0(G_y)$, all of whose irreducible representations are finite dimensional, the proof is identical to that of Proposition \ref{PropExN}.
\end{proof} 

In fact, the same argument shows that $\Delta(N)$ is the closure of the operator $A\otimes N + N\otimes A^{-1}$ (when suitably defined). Also the next proposition is proven in the same way.

\begin{Prop}\label{PropExDelT} There exists a unique regular operator \[T_{12}T_{13} \in M_2(C_0(G\underset{(0,1\rbrack}{\times} G))^{\eta}\] such that $M_2(C_c(G\underset{(0,1\rbrack}{\times} G))$ is a core for $T_{12}T_{13}$ and on which \[T_{12}T_{13} = \begin{pmatrix} A\otimes A & A\otimes N + N\otimes A^{-1} \\ 0 & A^{-1}\otimes A^{-1}\end{pmatrix}.\] 
\end{Prop}

\begin{Theorem}\label{TheoFieldG} There exists a unique  coassociative \[\Delta \in \Mor(C_0(G), C_0(G\underset{(0,1\rbrack}{\times} G))\] such that \[(\id\otimes \Delta)(T) = T_{12}T_{13}.\] 
\end{Theorem}

\begin{proof} As $(Q,T)$ generates $C_0(G)$ by Theorem \ref{TheoExT}, it is, by Proposition \ref{PropGen}, enough to prove that there exists, for each $0<q\leq 1$, a coassociative morphism \[\Delta_q: C_0(G_q) \rightarrow M(C_0(G_q)\otimes C_0(G_q))\] such that \begin{equation}\label{EqDelq}(\id\otimes \Delta_q)(T_q) = (T_q)_{12}(T_{q})_{13}.\end{equation}

For $0<q<1$, this is \cite[Section 4, Example 6]{Wor95} and \cite[Theorem 5.1]{PoWo90}.

For $q=1$, endow $G_1 = \C\times \R_0^+$ with the multiplication \[(n,a)(m,b) = (am+nb^{-1},ab). \] Then $G_1$ is a Lie group, and the functions $N_1,A_1$ are precisely the coordinate functions assiociated to the first and second variable. It is easily seen that the associated comultiplication satisfies \eqref{EqDelq} for $q=1$. 
\end{proof} 

Note that at $q=1$, we get the locally compact group \[G_1 = \left\{\begin{pmatrix} a & n \\ 0 & a^{-1}\end{pmatrix} \mid n\in \C,a>0\right\},\] while at $0<q<1$, we have that $(C_0(G_q), \Delta_q)$ is the discrete quantum group dual of Woronowicz's quantum $SU_q(2)$-group \cite{Wor87a}.

\subsection{Bisimplifiability} 

\begin{Def}(\cite[D\'{e}finition 4.1]{Bla96}) Let $A$ be a continuous field of C$^*$-algebras over a locally compact space $Y$, endowed with a coassociative morphism $\Delta \in \Mor(A,A\underset{C_0(Y)}{\otimes} A)$. One calls $\Delta$ \emph{bisimplifiable} if \[\lbrack \Delta(A)(1\otimes A) \rbrack = \lbrack (A\otimes 1)\Delta(A)\rbrack = A\underset{C_0(Y)}{\otimes} A,\] where $\lbrack \,\cdot\,\rbrack$ denotes closed linear span.  
\end{Def} 

Our aim will be to show that the field $C_0(G)$ together with its comultiplication constructed in Section \ref{SecFieldCom} is bisimplifiable. We will make use of the following lemma (which is implicitly used in \cite[Proposition 7.8]{Bla96}).

\begin{Lem}\label{LemDens} Assume $A$ is a $C_0(Y)$-algebra, and $B\subseteq A$ a $C_0(Y)$-invariant subspace. If the image $B_y$ of $B$ in $A_y$ is dense in $A_y$ for each $y\in Y$, then $B$ is dense in $A$.
\end{Lem} 
\begin{proof} Assume $B\neq A$. Then the set $C$ of functionals of norm $\leq 1$ which vanish on $B$ form a convex w$^*$-compact set with more than one point. Hence, by the Krein-Milman theorem, there exists a non-zero extreme functional $\omega$ in this set, necessarily of norm 1. 

We claim that $\omega$ factors over some $A_y$. Indeed, consider the absolute value $|\omega|\in A^*$, which we can extend to a continuous functional on $M(A) \subseteq A^{**}$. In particular, the restriction $|\omega|_{\mid C_0(Y)}$ induces a Radon probability measure $\mu$ on $Y$. If then the support of $\mu$ consists of more than one point, we can find Borel sets $Y_1,Y_2\subseteq Y$ with $Y_1\cap Y_2 = \emptyset$ and  $Y =  Y_1\cup Y_2$, and with $\mu(Y_i)\neq0$. Let $\chi_{S} \in C_0(Y)^{**} \subseteq A^{**}$ be the characteristic function of a Borel set $S$. We claim that then \[\omega = \mu(Y_1)\omega\left(\frac{\chi_{Y_1}}{\mu(Y_1)}\,\cdot\,\right) + \mu(Y_2)\omega\left(\frac{\chi_{Y_2}}{\mu(Y_2)}\,\cdot\,\right)\] is a non-trivial convex combination for $\omega$ within $C$. Indeed, the coefficients form a convex combination since \[1 = \|\omega\| = \||\omega|\| = \||\omega|_{\mid C_0(Y)}\| = \|\mu\| = \mu(Y).\] Further, since $B$ is a $C_0(Y)$-module, we have that $\omega(f\,\cdot\,)_{\mid B}= 0$ for any $f\in C_0(Y)^{**}$. Finally, since $C_0(Y)^{**}$ is central in $A^{**}$, it is easy to see that \[\|\omega(f\,\cdot\,)\|\leq |\omega|(|f|),\quad  \textrm{for any }f\in C_0(Y)^{**},\] and so the two functionals in the above convex combination are both in $C$. This contradicts the extremality of $\omega$. 

It follows from the above that $|\omega|$, and hence also $\omega$, factors over some $A_y$. But since $B_y$ is dense in $A_y$, this is impossible. 
\end{proof}

\begin{Theorem} The field $(C_0(G),\Delta)$ is bisimplifiable.
\end{Theorem}
\begin{proof}
As the localisations $(C_0(G_q),\Delta_q)$ are either a discrete quantum group or a locally compact group, it follows that the localisations are bisimplifiable. Hence, by Lemma \ref{LemDens} it suffices to show that \[\Delta(C_0(G))(1\otimes C_0(G)) \subseteq C_0(G\underset{(0,1\rbrack}{\times}G),\quad \Delta(C_0(G))(C_0(G)\otimes 1) \subseteq C_0(G\underset{(0,1\rbrack}{\times}G).\]

We will prove the first inclusion, the second one being similar. In fact, we will show that \[\Delta(C_c(H))(1\otimes C_c(H)) \subseteq C_c(G\underset{(0,1\rbrack}{\times} G),\] which is clearly sufficient to prove the statement. 

Let $0<q<1$, and let $\pi_n$ be the $n$th-dimensional representation of $C_0(G_q)$. Let $r\in \N_0$ and $0\leq l <r$, and let $\delta_{q;r,l}$ be the Dirac function in $C_c(H_q)$ at the point $(q^r,q^{\frac{1-r}{2}+l})$. Then it is well known, from the fusion rules for representations of $C_0(G_q)$, that $(\pi_n\otimes \pi_m)\Delta_q(\delta_{q;r,l}) = 0$ unless $r = n+m-2k-1$ with $0\leq k\leq \min\{n,m\}-1$, in which case it is a one-dimensional projection. In the latter case, choose a unit vector $\xi_{q;r,l}^{(n,m)}$ in its range. From the concrete formula $\Delta_q(A_q)= A_q\otimes A_q$, it follows that we must have \[\xi_{q;r,l}^{(n,m)} = \sum_{p = \max\{0,l-m\}}^{\min\{n,l\}} f_{q;r,l}^{(n,m)}(p) e_{p}^{(n)}\otimes e_{l-p}^{(m)},\] and then \[\Delta_q(\delta_{q;r,l}) =  \underset{|n-m|+1\leq r\leq n+m-1}{\underset{n+m-r-1 \textrm{ even}}{\sum_{n,m}}}\sum_{p,p' = \max\{0,l-m\}}^{\min\{n,l\}} \overline{f_{q;r,l}^{(n,m)}}(p)f_{q;r,l}^{(n,m)}(p') e_{p',p}^{(n)} \otimes e_{l-p',l-p}^{(m)}.\]

Pick now $0<\varepsilon <1$, and suppose $f,g\in C_c(H)$ with $f(q,t,a)=g(q,t,a)=0$ for $\min\{q,t\}\leq \varepsilon$. Then for $q\leq \varepsilon$, $f_q,g_q$ are zero, while for $\varepsilon< q <1$ they are linear combinations of the $\delta_{q;r,l}$ with $1\leq r< \log_q(\varepsilon)$ and $0\leq l <r$. Hence $\Delta_q(f_q)(1\otimes g_q)$ is zero for $q\leq \varepsilon$, and for $\varepsilon<q<1$ an (infinite) linear combination of elements $e_{p',p}^{(n)}\otimes e_{l-p',l-p}^{(m)}$ where $n,m$ satisfy $m< \log_q(\varepsilon)$ and $|n-m|+1< \log_q(\varepsilon)$.  In particular, no terms appear in which $n\geq 2\log_q(\varepsilon)$. 

It follows that if $h$ is a function with compact support on $H$ and such that $h(q,t,a) = 1$ for $\min\{q,t\}\geq \varepsilon^2$, then \[\Delta_q(f_q)(1\otimes g_q)=\Delta_q(f_q)(h_q\otimes g_q),\quad 0<q<1.\] By continuity of the field, this implies $\Delta(f)(1\otimes g) = \Delta(f)(h\otimes g)$, which finishes the proof. 
\end{proof} 

\subsection{Invariant integral on $G$}

In this section, we prove that $(C_0(G),\Delta)$ has an associated continuous field of multiplicative unitaries. 

Recall the $C_0((0,1\rbrack)$-linear map $\varphiH$ on $C_c(I)$ defined in Lemma \ref{LemContFieldHilb}. By restriction, we can consider it as a map on $C_c(H)$. Let further $F: C_0(G) \rightarrow C_0(H)$ be the canonical conditional expectation.

\begin{Lem} There exists a unique $C_c((0,1\rbrack)$-linear functional \[\varphiH': C_c(G) \rightarrow C_c((0,1\rbrack)\] such that $\varphiH'(x) = \varphiH(F(x))$ for all $x\in C_c(G)$.
\end{Lem} 
\begin{proof} As $F(x) \in C_c(H)$ for $x\in C_c(G)$, we can simply define $\varphiH' = \varphiH\circ F$. 
\end{proof} 

In a similar way, we can define the $C_c((0,1\rbrack)$-linear map \[(\psi'\otimes \id): C_c(G\underset{(0,1\rbrack}{\times} G) \rightarrow C_c(G).\] 

\begin{Lem}\label{LemInvPsi} Let $f,g\in C_c(G)$. Then  \begin{equation}\label{EqInvPsi}(\varphiH'\otimes \id)(\Delta(f)(1\otimes g))  = \varphiH(f)g.\end{equation}
\end{Lem} 
\begin{proof} By construction, we have that \[((\varphiH'\otimes \id)(\Delta(f)(1\otimes g)))_q = (\varphiH_q\circ F_q\otimes \id)(\Delta_q(f_q)(1\otimes g_q)),\qquad 0<q\leq 1.\] But one easily verifies that, for $0<q<1$, the functional $\varphiH_q'$ coincides with the right invariant integral for $(C_0(G_q),\Delta_q)$, hence the right hand side equals $\varphiH_q'(f_q)g_q = (\varphiH'(f)g)_q$ for $0<q<1$. By continuity (or a direct verification), this also holds at $q=1$. Hence \eqref{EqInvPsi} holds. 
\end{proof} 

Consider now the Hilbert $C_0((0,1\rbrack)$-module $\mathscr{F}$ obtained by completing $C_c(G)$ with respect to the inner product \[\langle x,y\rangle = \varphiH'(x^*y).\] Then the left $C_0(G)$-module structure on $C_c(G)$ extends to a $^*$-representation of $C_0(G)$ on $\mathscr{F}$. 

Similarly, we can consider the completion of $C_c(G\underset{(0,1\rbrack}{\times} G)$ with respect to $\varphiH'\otimes \varphiH' = \varphiH'\circ (\varphiH'\otimes \id)$, and the Hilbert module completion can be identified with $\mathscr{F}\underset{C_0((0,1\rbrack)}{\otimes}\mathscr{F}$. 

\begin{Theorem} There exists a unique continuous field of multiplicative unitaries \[V \in \mathcal{L}(\mathscr{F}\underset{C_0((0,1\rbrack)}{\otimes}\mathscr{F})\] such that \[V(x\otimes y) = \Delta(x)(1\otimes y),\quad x,y\in C_c(G).\]
\end{Theorem}

The fact that $V$ is a continuous field of multiplicative unitaries simply means that $V$ satisfies the pentagon equation, see \cite[D\'{e}finition 4.6]{Bla96}.

\begin{proof} By Lemma \ref{LemInvPsi}, we find that $V$ exists as an isometry. As the localisations $V_q$ coincide with the right multiplicative unitaries of the locally compact quantum groups $(C_0(G_q),\Delta_q)$, it follows that the $V_q$ are unitaries. Hence, using that $\mathcal{K}(\mathscr{F}\underset{C_0((0,1\rbrack)}{\otimes} \mathscr{F})$ is a $C_0(Y)$-algebra with fibers the compact operators on the $\mathscr{F}_y\otimes \mathscr{F}_y$, we find that\[p = 1-VV^* \in M(\mathcal{K}(\mathscr{F}\underset{C_0((0,1\rbrack)}{\otimes} \mathscr{F}))\] satisfies $p_y = 0$ for all $y$. Hence $p=0$, and $V$ a unitary. 

The pentagon equation for $V$ follows immediately from its definition. 
\end{proof}

\end{document}